\documentclass[a4paper,11pt]{amsart}

\usepackage{fouriernc}
\usepackage[T1]{fontenc}

\usepackage{graphicx}
\usepackage{amsmath}
\usepackage{amssymb}
\usepackage[active]{srcltx}
\usepackage{hyperref}
\usepackage{bbm}

\usepackage{color}
\definecolor{red}{rgb}{1,0,0}

\definecolor{gre}{rgb}{0,0.5,0}

\definecolor{blu}{rgb}{0,0,1}

\definecolor{pur}{rgb}{0.6,0.16,0.9}

\usepackage{geometry}
\newtheorem{prop}{Proposition}%

\theoremstyle{definition}

\theoremstyle{remark}
\theoremstyle{plain}

\def\CC{{\mathbb C}}

\def\NN{{\mathbb N}}
\def\QQ{{\mathbb Q}}

\def\ZZ{{\mathbb Z}}

\def\e{\mathrm{e}}
\def\i{\mathrm{i}}

\def\vol{\operatorname{vol}}

\title{Delone sets generated by square roots}
\author{Jens Marklof}
\address{Jens Marklof, School of Mathematics, University of Bristol, Bristol BS8 1TW, U.K.\newline \rule[0ex]{0ex}{0ex} \hspace{8pt}{\tt j.marklof@bristol.ac.uk}}
\date{31 May 2019/21 Feb 2020; to appear in American Math. Monthly}
\subjclass[2010]{11K06, 52C15}

\begin{document}

\begin{abstract}
Delone sets are locally finite point sets, such that (a) any two points are separated by a given minimum distance, and (b) there is a given radius so that every ball of that radius contains at least one point. Important examples include the vertex set of Penrose tilings and other regular model sets, which serve as a mathematical model for quasicrystals. In this note we show that the point set given by the values $\sqrt n\, \e^{2\pi\i \alpha \sqrt n}$ with $n=1,2,3,\ldots$ is a Delone set in the complex plane, for any $\alpha>0$. This complements Akiyama's recent observation (see arxiv.org/abs/1904.10815) that $\sqrt n\, \e^{2\pi\i \alpha n}$ with $n=1,2,3,\ldots$ forms a Delone set, if and only if $\alpha$ is badly approximated by rationals. A key difference is that our setting does not require Diophantine conditions on $\alpha$.
\end{abstract}

\maketitle

Consider an infinite sequence of real numbers $\xi_1,\xi_2,\xi_3,$ etc. We are interested in the patterns  generated by the complex numbers
\begin{equation}
z_n = \sqrt{n}\; \e^{2\pi\i \xi_n},
\end{equation}
see Figure \ref{fig1}.
Before discussing the specific case $\xi_n=\alpha\sqrt n$, we keep things general and explore the relationship between the distribution of the sequence $(\xi_n)_n$ mod 1 and the distribution of $(z_n)_n$ in the complex plane $\CC$. 

A sequence $(\xi_n)_n$ of real numbers is said to be {\em uniformly distributed mod $1$} if for any real numbers $a < b$ with $b-a\leq 1$ we have that
\begin{equation}\label{UD}
\lim_{N\to\infty} \frac{\#\{ n\leq N : \xi_n\in [a,b)+\ZZ \}}{N} = b-a. 
\end{equation}
Classic examples of sequences that are uniformly distributed mod $1$ are $\xi_n=\alpha n^k$ for any $\alpha\not\in\QQ$ and $k\in\NN$, and $\xi_n=\alpha n^\beta$ for any $\alpha\neq 0$ and non-integer $\beta>0$; see Theorem 3.2 and Exercise 3.9 in \cite[Ch.~ 1]{KN}. A sequence $z_1,z_2,\ldots$ of complex numbers is said to have {\em asymptotic density $\rho$}, if for any bounded set $B$ with boundary of Lebesgue measure zero, we have
\begin{equation}\label{AD}
\lim_{R\to\infty} \frac{\#\{ n\in\NN : z_n\in RB \}}{R^2} = \rho \vol B,
\end{equation}
where $\vol$ denotes the Lebesgue measure in $\CC$. 

\begin{figure}
\begin{center}
\includegraphics[width=0.49\textwidth]{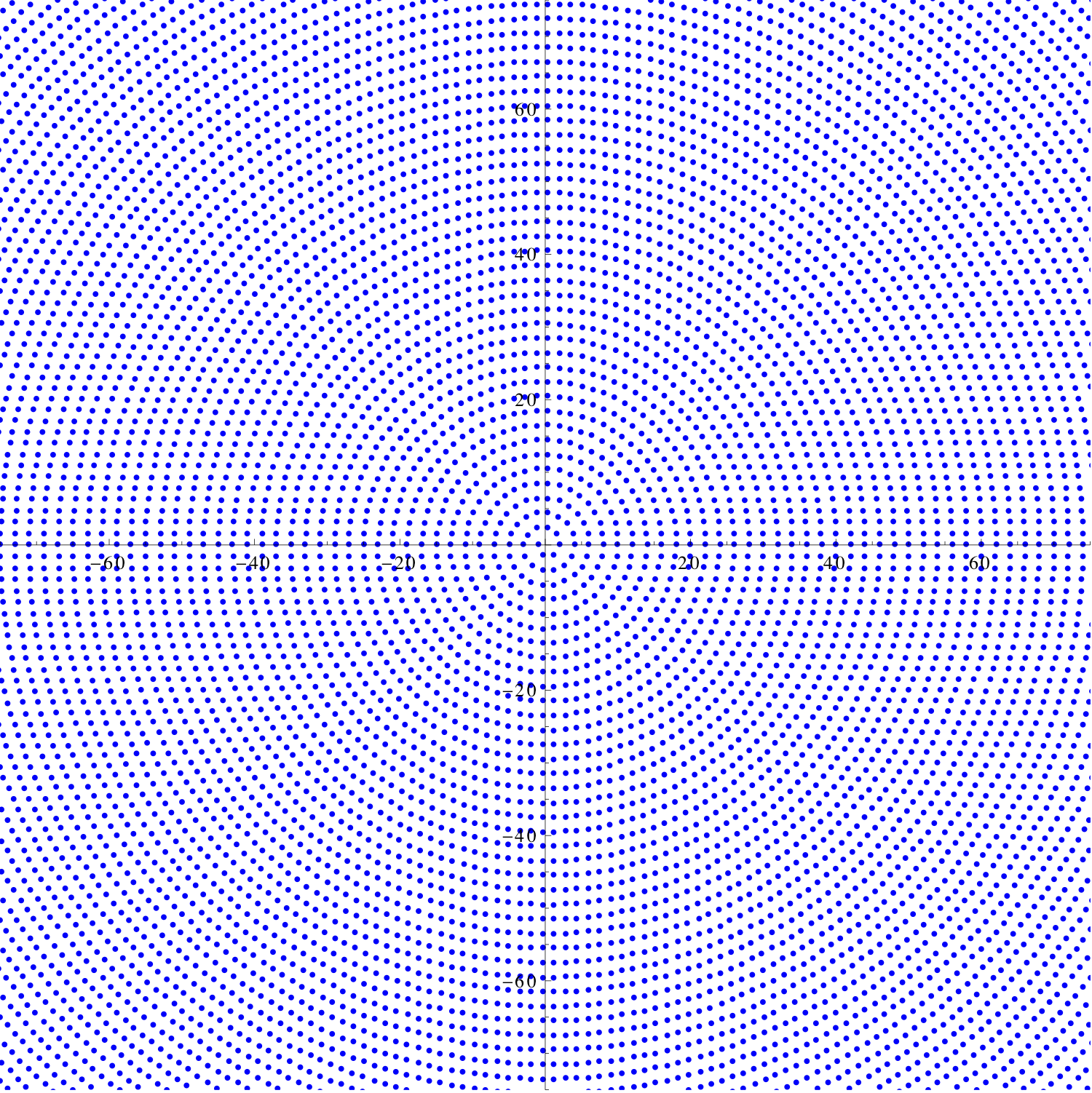}
\includegraphics[width=0.49\textwidth]{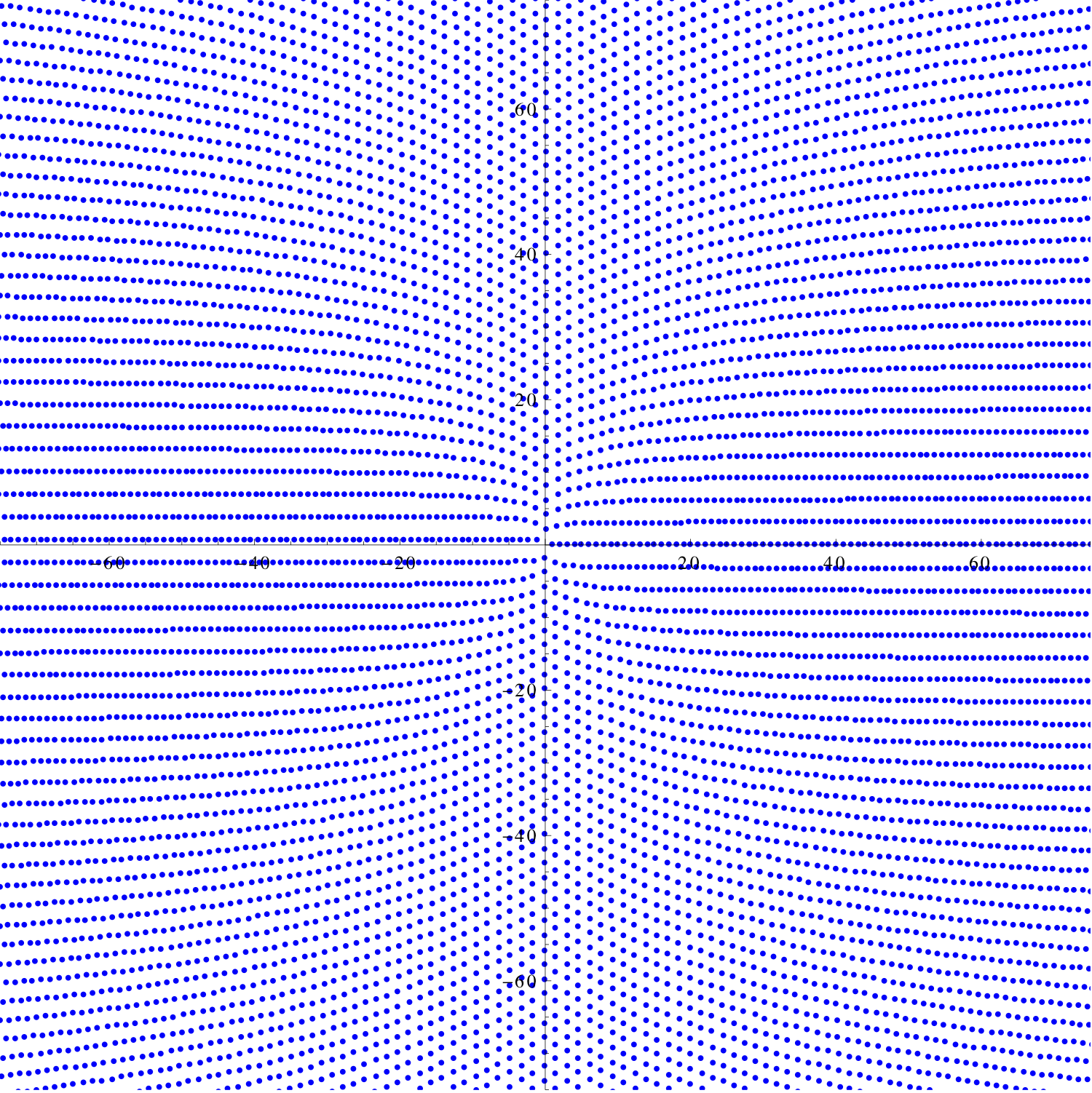}
\includegraphics[width=0.49\textwidth]{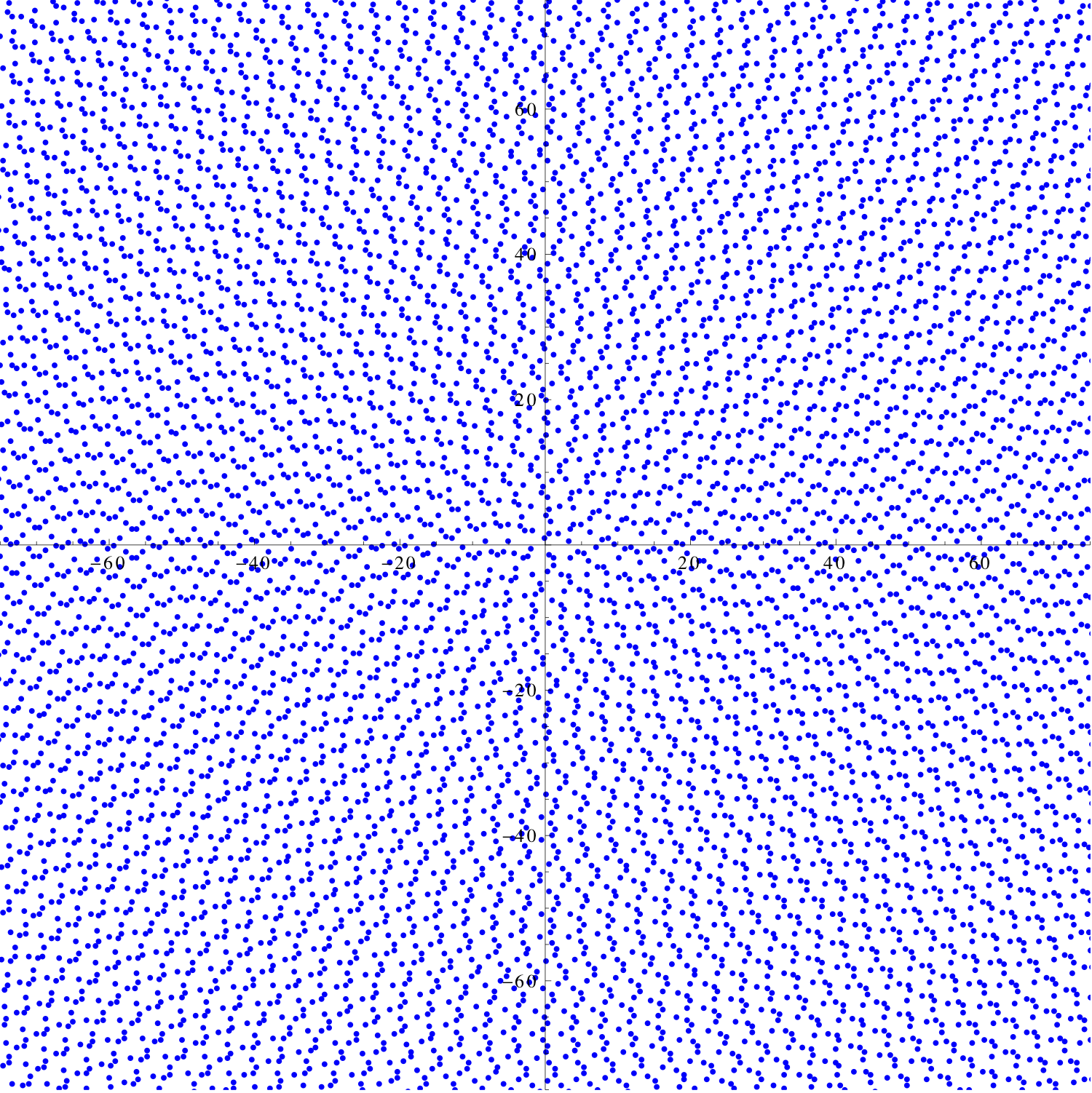}
\includegraphics[width=0.49\textwidth]{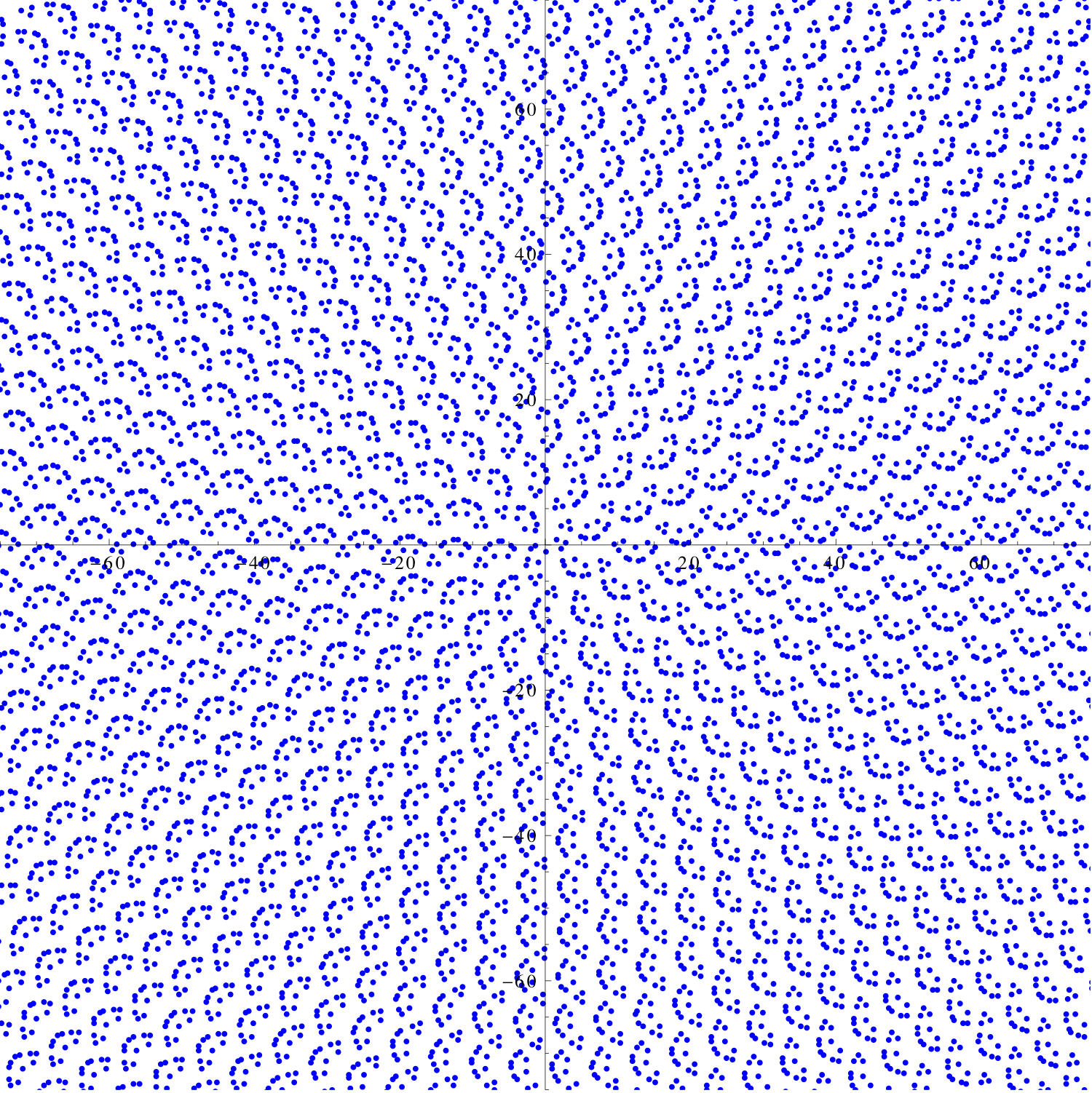}
\end{center}
\caption{The point sequence $(z_n)_n$ with $\xi_n=\alpha \sqrt{n}$ and $\alpha=\frac12,1,\sqrt{\pi},\sqrt{\frac{1+\sqrt 5}{2}}$ (clockwise, from top left).} \label{fig1}
\end{figure}

\begin{figure}
\begin{center}
\includegraphics[width=0.49\textwidth]{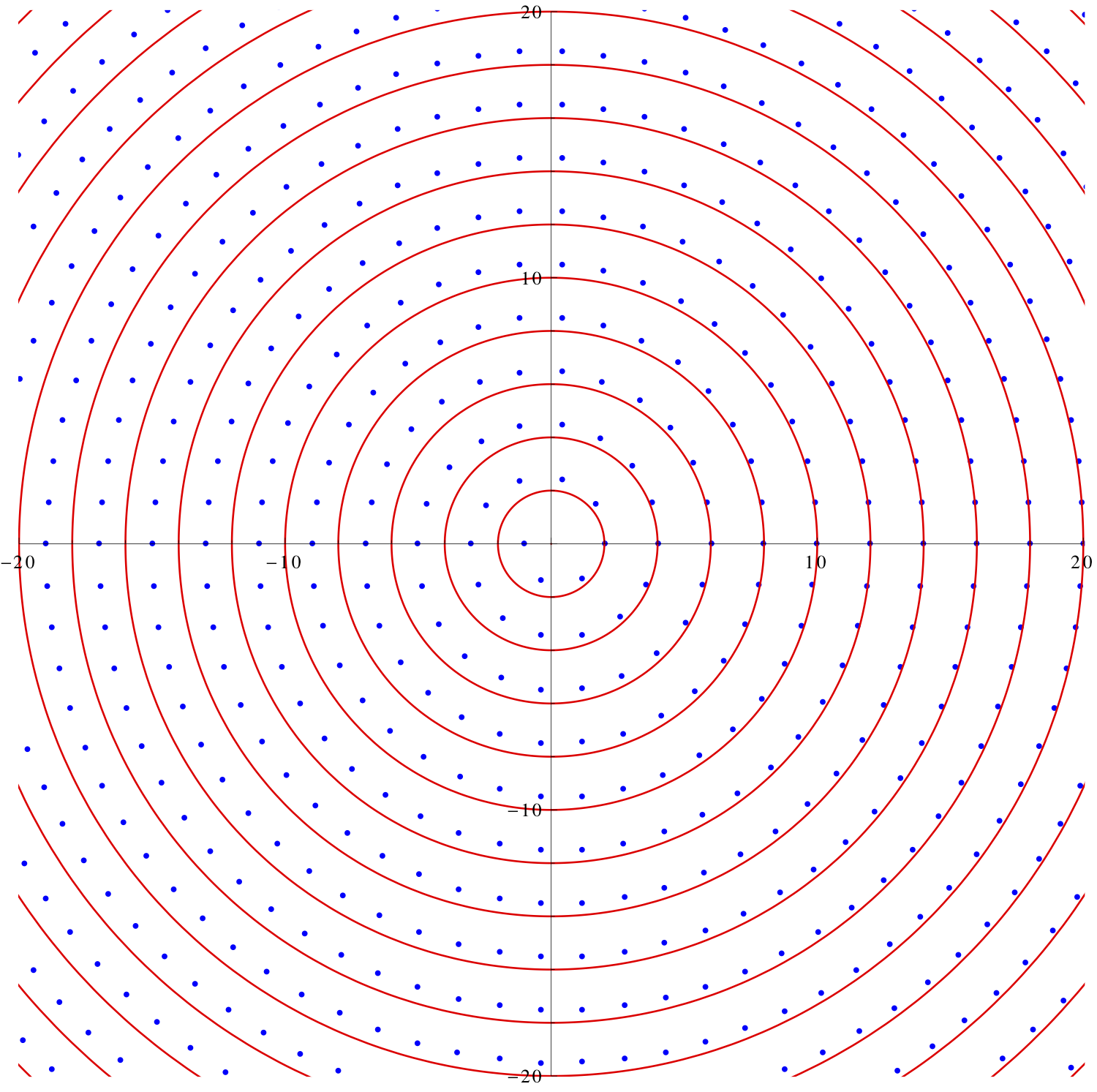}
\includegraphics[width=0.49\textwidth]{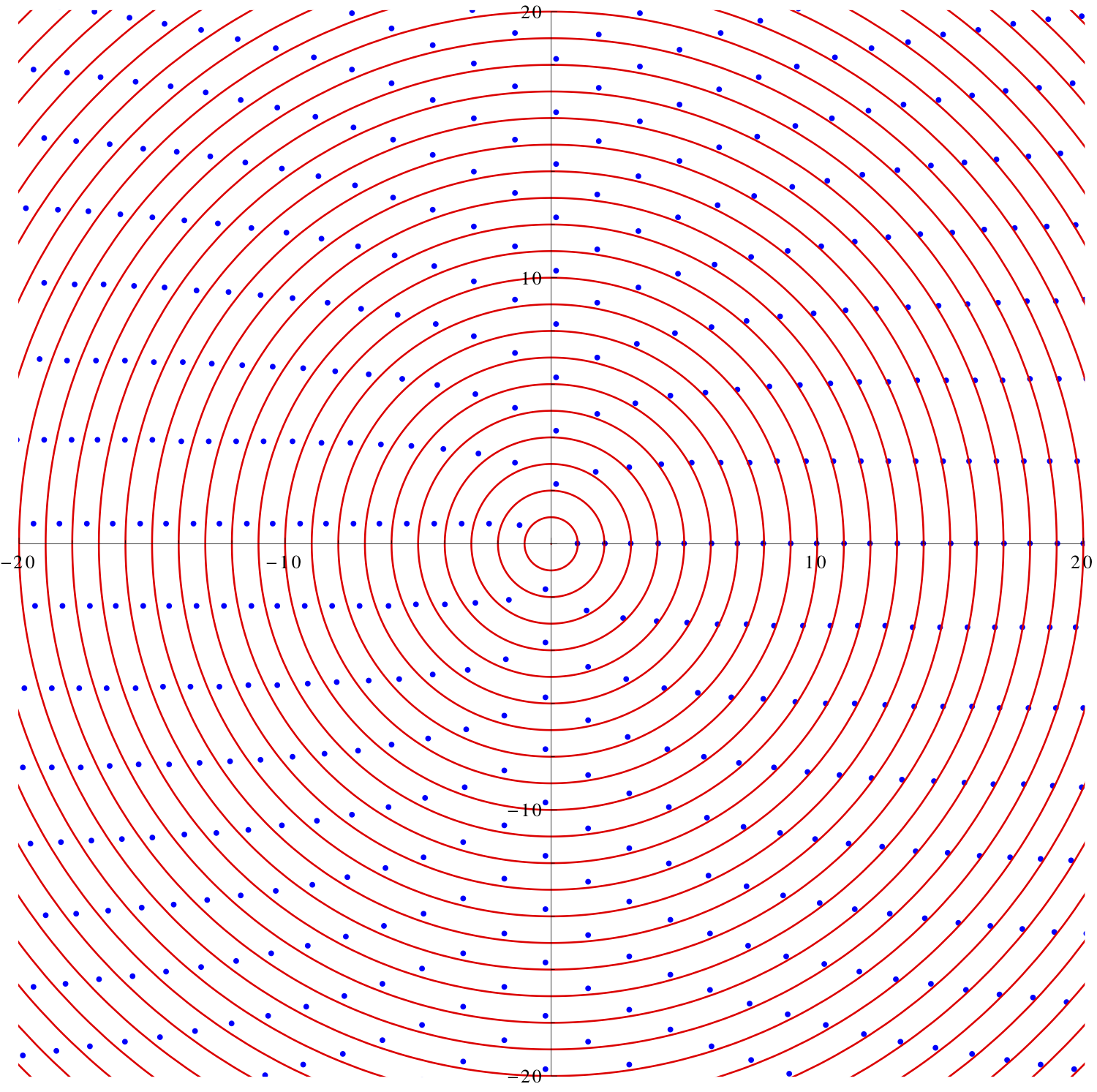}
\includegraphics[width=0.49\textwidth]{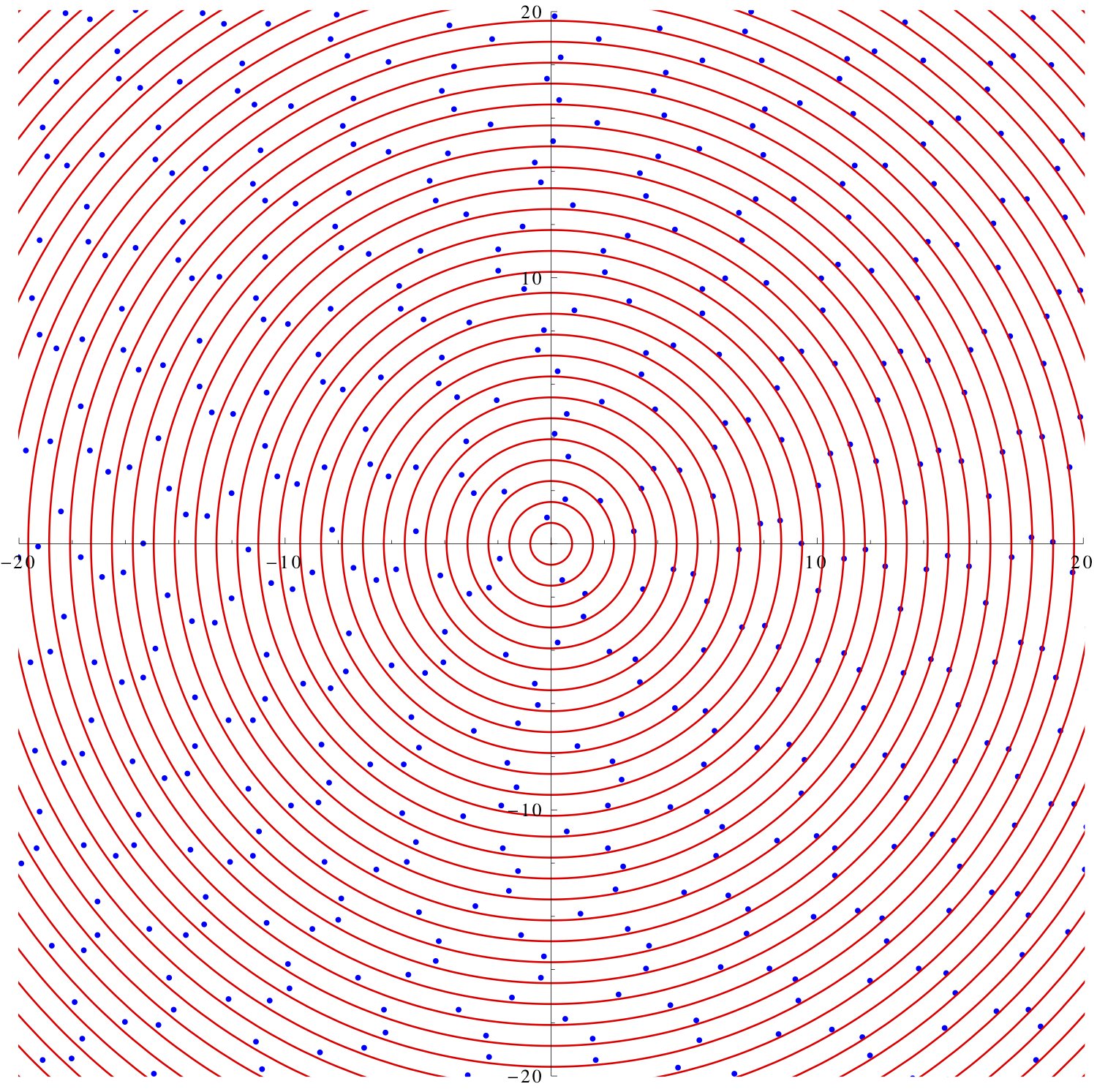}
\includegraphics[width=0.49\textwidth]{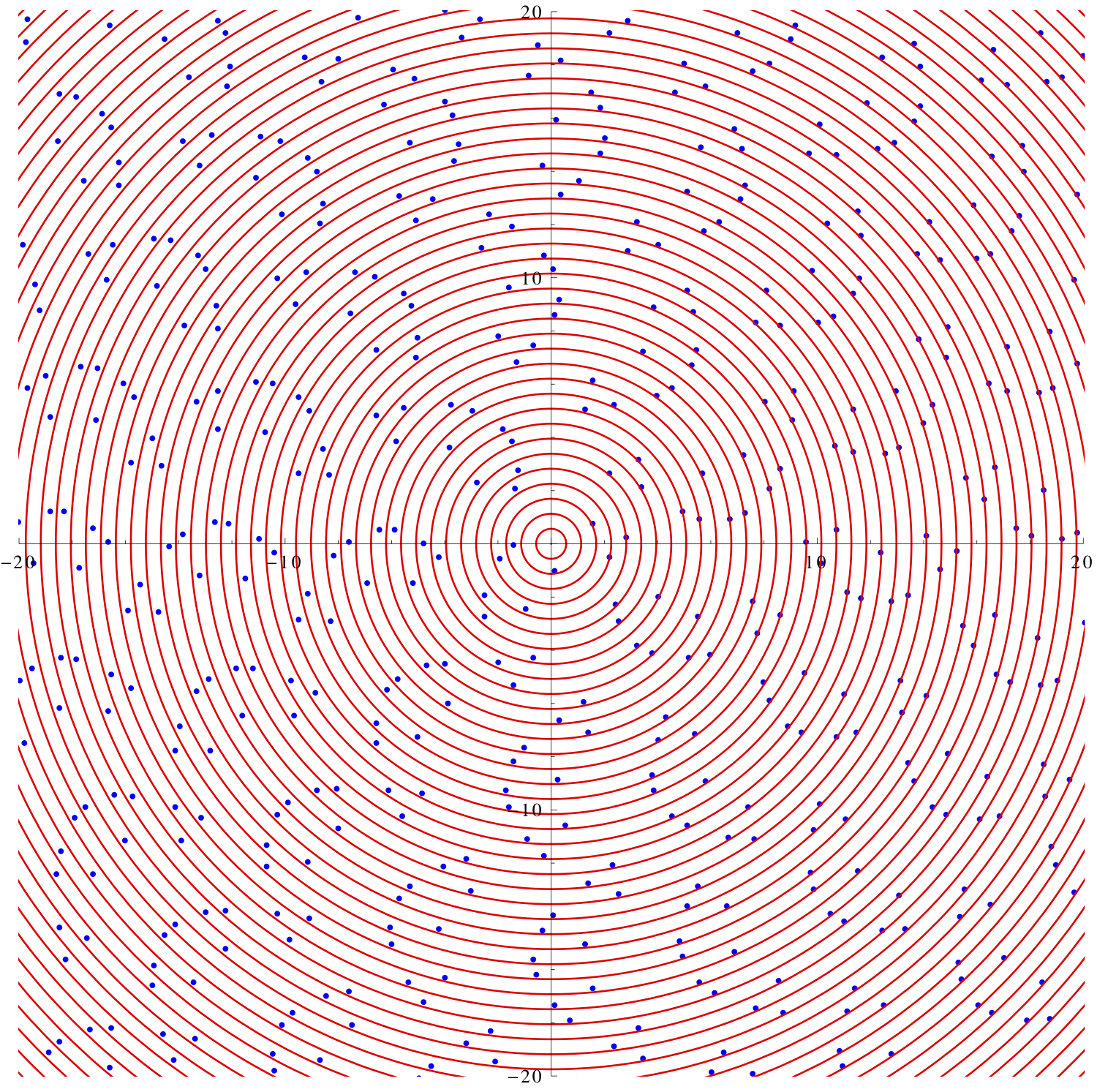}
\end{center}
\caption{A zoom-in of each of the point patterns in Figure \ref{fig1}, overlayed with annuli of width $h=1/\alpha$. The points in an annulus of radius $R$ are approximately spaced by $2\pi R\times\frac{\alpha}{2R}=\pi\alpha$; see \eqref{deltaj}.} \label{fig2}
\end{figure}

Our first observation is the following. 
\begin{prop}\label{Adens}
Let $(\xi_n)_n$ be a sequence of real numbers and let $(z_n)_n$ be the corresponding sequence given by $z_n = \sqrt{n}\; \e^{2\pi\i \xi_n}$. Then $(z_n)_n$ has asymptotic density $\rho=\pi^{-1}$ if and only if $(\xi_n)_n$ is uniformly distributed mod $1$.
\end{prop}

\begin{proof}
Consider the sectors
\begin{equation}
Z_{a,b,c,d} = \{ z\in\CC : \tfrac{1}{2\pi} \arg z \in [a, b) +\ZZ ,\;  c\leq |z|< d\}
\end{equation}
with $b-a\leq 1$. Note that for $R>0$ we have $RZ_{a,b,c,d}=Z_{a,b,Rc,Rd}$ and $\vol Z_{a,b,0,R}=(b-a) \pi R^2$. Hence, \eqref{AD} with $\rho=\pi^{-1}$ and the choice $B=Z_{a,b,0,1}$, implies \eqref{UD}. 

The reverse implication is a little more involved. Let us assume \eqref{UD} holds. Then, by the previous argument, \eqref{AD} holds for $B=Z_{a,b,0,d}$ for any $d>0$. Taking the set $Z_{a,b,c,d}=Z_{a,b,0,d}-Z_{a,b,0,c}$ we see that \eqref{AD} also holds for $B=Z_{a,b,c,d}$ for any $0\leq c<d$. 
A standard argument for Jordan measurable sets shows that any bounded set $B$ with boundary of Lebesgue measure zero can be approximated arbitrarily well by unions of sectors of the form $Z_{a,b,c,d}$. This establishes that indeed \eqref{UD} implies \eqref{AD}, and our claim is proved. 
\end{proof}

Let us now turn to the question under which conditions on $(\xi_n)_n$ the sequence $(z_n)_n$ forms a Delone set. We say a point set is {\em uniformly discrete} if any two points are separated by a given minimum distance, and {\em relatively dense} if there is a given radius so that every ball of that radius contains at least one point. A {\em Delone set} is a point set that satisfies both of these properties. 

Given $M$ ordered real numbers $\zeta_1\leq \cdots\leq\zeta_M$ in the unit interval $[0,1]$, we define the corresponding gaps by $\delta_j=\zeta_{j+1}-\zeta_j$ for $j=1,\ldots,M-1$. We furthermore set $\delta_M=\zeta_1+1-\zeta_M$, which represents the gap between $\zeta_1$ and $\zeta_M$, after gluing the end points of the interval $[0,1]$. The numbers $\delta_1,\ldots,\delta_M$ thus describe the gaps of $\zeta_1,\ldots,\zeta_M$ mod $1$. 

Recall that the {\em fractional part} of a real number $x$ is defined by $x-\lfloor x \rfloor$, where $\lfloor x \rfloor$ is the largest integer less than or equal to $x$. Given $(\xi_n)_n$, $h>0$ and $R>0$, we denote by $g_R^h$  the minimal gap mod $1$ (as defined above) between the fractional parts of the elements $\xi_n$ with $n$ ranging over $R^2\leq n< (R+h)^2$. 
Note that for large $R$ (and $h$ fixed) there are approximately $2hR$ such elements, and hence the average gap size between the fractional parts is about $1/2hR$. 
We have the following two facts.

\begin{prop} \label{UDRD}
Let $(\xi_n)_n$ be a sequence of real numbers and let $(z_n)_n$ be the corresponding sequence given by $z_n = \sqrt{n}\; \e^{2\pi\i \xi_n}$. Then:
\begin{enumerate}
\item[(i)] $(z_n)_n$ is uniformly discrete if and only if there exists $h>0$ such that $$\inf_{R\geq 1} \, R g_R^h >0.$$
\item[(ii)] $(z_n)_n$ is relatively dense if and only if there exists $h>0$ such that $$\sup_{R\geq 1} \, R g_R^h <\infty.$$
\end{enumerate}
\end{prop}

\begin{proof}
Consider the annulus $A_R^h=\{ z\in\CC: R \leq |z| < R+h \}$ with inner radius $R\geq 1$ and width $h>0$. To prove (i), let us first assume that there exists $h>0$ such that $\inf_{R\geq 1} \, R g_R^h > 0$, and establish uniform discreteness. Set $\delta=\min\{\frac12,\inf_{R\geq 1} \, R g_R^h \}$ and note $0<\delta\leq 1/2$. Then all points $z_n$ in $A_R^h$ are separated by a distance at least $2R \sin(\pi\delta/R)$, which in turn is bounded below by $4 \delta$ since $2x\leq \sin(\pi x)$ for $0\leq x\leq 1/2$. 
This lower bound is independent of $R$, and hence $(z_n)_n$ is uniformly discrete. For the reverse implication in (i), assume $(z_n)_n$ is uniformly discrete, i.e.\ there is some $r>0$ so that the distance between any two points in $(z_n)_n$ is greater than $r$. If these points are in the same annulus $A_R^h$ with $h<r$, then the difference of their arguments is bounded below by the smaller of the two numbers $(r-h)/R$ and $r/(R+h)$. Thus $2\pi R g_R^h\geq \min\{ r-h, r/(1+h)\}$ for all $R\geq 1$. For the choice $h=r/2$, say, we therefore obtain $\inf_{R\geq 1} R g_R^h>0$, which completes the proof of claim (i). The proof of statement (ii) follows from a similar argument.
\end{proof}

Let us now turn to our example $\xi_n=\alpha\sqrt{n}$, with $\alpha>0$ arbitrary. We will show that in this case the following holds.

\begin{prop}\label{Delo}
Let $\alpha>0$ and $\xi_n=\alpha\sqrt{n}$. Then $(z_n)_n$ is a Delone set with asymptotic density $\pi^{-1}$.
\end{prop}
 
\begin{proof}
The asymptotic density follows from Proposition \ref{Adens}, since the sequence $(\xi_n)_n$ is uniformly distributed mod $1$ (this is a special case of the examples discussed in the introduction). To prove the upper and lower bounds on the gaps mod $1$ needed in Proposition \ref{UDRD}, let us label the integers $n$ in the range $R^2\leq n< (R+h)^2$ by $m+1,\ldots,m+M$. Note that, since $n$ ranges over a half-open interval of length $2hR + h^2$, the number $M$ of points in that interval satisfies the bound
\begin{equation}\label{Mrange}
2hR + h^2 \leq M < 2hR+h^2+1.
\end{equation}
We have $\xi_{m+M}-\xi_{m+1}<\alpha (R+h)-\alpha R=\alpha h$. Any choice $h\leq 1/\alpha$ thus ensures that $\xi_{m+M}-\xi_{m+1}<1$. This in turn implies that the the numbers $\xi_{m+1},\ldots, \xi_{m+M}$ are ordered mod $1$ and hence the gaps are  given by $\delta_j=\xi_{m+j+1}-\xi_{m+j}$ for $j=1,\ldots,M-1$, and $\delta_M=\xi_{m+1}+1-\xi_{m+M}$. Note in particular that $\delta_M>0$. Examples of annuli of width $h=1/\alpha$ are illustrated in Figure \ref{fig2}, where one can join the dots in each annulus through a full rotation. We have, for $j=1,\ldots,M-1$,
\begin{equation}
\delta_j=\frac{\alpha}{\sqrt{m+j+1}+\sqrt{m+j}} ,
\end{equation}
and so
\begin{equation}\label{deltaj}
\frac{\alpha}{2(R+h)} < \delta_j \leq  \frac{\alpha}{2R} ,
\end{equation}
which provides the required uniform upper and lower bounds for the first $M-1$ gap sizes for any positive $h\leq 1/\alpha$. (In fact, we see that for $R$ large, the rescaled gap $R \delta_j$ is approximately $\alpha/2$, i.e., independent of $j$ and $R$. This explains the regular spacing in each annulus in Figure \ref{fig2}.)

We also need the corresponding bounds for the gap $\delta_M$, which we write as $\delta_M=1-(\delta_1+\cdots+\delta_{M-1})$. In view of \eqref{deltaj} and \eqref{Mrange}, we have for $h=1/\alpha$
\begin{equation}
\delta_M \leq  1-(M-1)\frac{\alpha}{2(R+h)} \leq \frac{\alpha+\alpha^{-1}}{2(R+h)} < \frac{\alpha+\alpha^{-1}}{2R},
\end{equation}
and hence $R \delta_M$ also has a uniform upper bound. We furthermore have the lower bound 
\begin{equation}
\delta_M \geq  1- \frac{(M-1)\alpha}{2R} > 1-h\alpha-\frac{h^2 \alpha}{2R}  .
\end{equation}
Any choice of $h<1/\alpha$ will work for our purposes; let us take $h=1/2\alpha$. Then, for all $R\geq 2+ h/2$ we have $R \delta_M > 1$. In the remaining range $1\leq R< 2+ h/2$ there are only finitely many values for $\delta_M$. This establishes the desired uniform lower bound for all $R\geq 1$.

In summary, we have shown that (i) $\inf_{R\geq 1} \, R g_R^h >0$ for $h=1/2\alpha$, and (ii) $\sup_{R\geq 1} \, R g_R^h <\infty$ for $h=1/\alpha$. The claim now follows from Proposition \ref{UDRD}. 
\end{proof}

It is interesting to note that, unlike the case of $n\alpha$ mod $1$ considered in \cite{Akiyama19} (with $\alpha$ badly approximated by rationals), the overall fine-scale distribution of $\alpha\sqrt n$ mod $1$ does not display a bounded gap phenomenon. For $\alpha^2\in\QQ$, the gap distribution is well understood and indeed arbitrarily large and small gaps (on the scale of the average gap) will occur \cite{Elkies04}. In the case of typical $\alpha^2\notin\QQ$, the fine-scale statistics are conjectured to follow a Poisson law, with an exponential gap distribution. This conjecture is based on numerical and heuristic evidence \cite{Elkies04,Marklof07}, but its proof is still an open problem. Remarkably, the two-point statistics of $\xi_n=\sqrt n$ mod $1$ is Poisson \cite{ElBaz15} (if one removes the density-zero subset of indices $n$ that are perfect squares), even though the gap distribution and other local statistics are not Poisson \cite{Elkies04}.


\end{document}